\documentclass[11pt]{article}
\usepackage{amsmath,amsxtra,amssymb,color,latexsym,epsfig,amscd,amsthm,subfigure,fancybox,epsfig}
\usepackage[mathscr]{eucal}
\usepackage{bbm}
\usepackage{graphicx}
\usepackage{enumerate}
\usepackage{enumitem}
\usepackage{epsfig} 
\usepackage{epstopdf}
\usepackage{cases}
\newtheorem{thm}{Theorem}[section]
\newtheorem {asp}{Assumption}[section]
\newtheorem{lm}{Lemma}[section]

\newtheorem{defn}{Definition}[section]

\newtheorem{rem}{Remark}[section]
\newtheorem{exm}{Example}[section]

\DeclareMathOperator{\trace}{tr}

\newcommand{\eps}{\varepsilon}

\newcommand{\B}{\mathcal{B}}

\newcommand{\p}{\mathfrak{p}}
\newcommand{\m}{\mathfrak{m}}
\newcommand{\C}{\mathcal{C}}
\newcommand{\D}{\mathcal{D}}
\newcommand{\DD}{\mathbb{D}}
\newcommand{\M}{\mathcal{M}}
\newcommand{\F}{\mathcal{F}}

\newcommand{\E}{\mathbb{E}}

\newcommand{\LL}{\mathcal{L}}

\newcommand{\N}{{\mathbb{Z}}_+}

\newcommand{\Z}{\mathbb{Z}}

\newcommand{\PP}{\mathbb{P}}

\newcommand{\R}{\mathbb{R}}

\newcommand{\BF}{\mathbb{F}}

\numberwithin{equation}{section}
\newcommand{\1}{\boldsymbol{1}}

\newcommand{\nz}{{n_0}}

\newcommand{\cd}{(\cdot)}
\newcommand{\beq}[1]{\begin{equation} \label{#1}}
\newcommand{\eeq}{\end{equation}}

\newcommand{\wdt}{\widetilde}

\newcommand{\bed}{\begin{displaymath}}
\newcommand{\eed}{\end{displaymath}}
\newcommand{\bea}{\bed\begin{array}{rl}}
\newcommand{\eea}{\end{array}\eed}
\newcommand{\ad}{&\!\!\!\disp}

\newcommand{\barray}{\begin{array}{ll}}
\newcommand{\earray}{\end{array}}

\def\disp{\displaystyle}
\newcommand{\rr}{{\Bbb R}}

\def\bar{\overline}
\def\hat{\widehat}
\def\a.s{\text{\;a.s.\;}}

\def\bnu{\boldsymbol{\nu}}

\begin{document}

\title{Recurrence for linearizable switching diffusion with past dependent switching and countable state space\thanks{This
research was supported in part by the
Air Force Office of Scientific Research under FA9550-15-1-0131.}}
\author{Dang Hai  Nguyen\thanks{Department of Mathematics, Wayne State University, Detroit, MI
48202,
dangnh.maths@gmail.com.} \and
George Yin\thanks{Department of Mathematics, Wayne State University, Detroit, MI
48202,
gyin@math.wayne.edu.}}
\maketitle

\maketitle

  \centerline{In honor of Jiongmin Yong on the occasion of  his 60th Birthday}   

\bigskip

 \centerline{(Communicated by)}

\begin{abstract}
This work continues
 and substantially extends
 our recent work on switching diffusions with the switching
 processes
  that depend on the past states and that take
   values in a countable state space.
That is, the discrete
components of the two-component processes take values in a countably infinite set and its switching rates at
 current time
 depend on the current value of the continuous component.
This paper focuses on recurrence,
positive recurrence, and weak stabilization of such systems.
In particular,
the paper
aims to providing more verifiable conditions on recurrence and positive recurrence and related issues.
Assuming that the system is linearizable, it provides  feasible conditions focusing
on the coefficients of the systems
for positive recurrence.
Then  linear feedback controls for weak stabilization are considered.
Some illustrative examples are also given.


\end{abstract}

\newpage
\section{Introduction}\label{sec:int}
In the new era, systems arising in engineering practice, social network, biological and medical applications, and ecological modeling etc. demand
more feasible models.
 Owing to the pressing need, significant effort has been devoted to developing  sophisticated system formulations for control, optimization, and stability of systems.
Emerging and existing applications in  wireless communications, queueing networks, biological models,
ecological systems, financial engineering, and social networks demand the
mathematical modeling, analysis, and computation
of hybrid models in which continuous dynamics and discrete events coexist.
A switching diffusion is a two-component process $(X(t),\alpha(t))$, a continuous component and a discrete component taking values in a discrete set (a set consisting of isolated points).
When the discrete component takes a value $i$ (i.e., $\alpha(t)=i$),
the continuous component $X(t)$ evolves according to the diffusion process whose drift and diffusion coefficients depend on $i$.
Because of their importance, many works have been devoted to such hybrid dynamic systems; see \cite{KZY,SX,ZY2,ZY,YZW,ZWYJ} and the references therein, and also \cite{YinZ10} for a comprehensive study to the
 Markov process  $(X(t),\alpha(t))$
in which the generator of $\alpha(t)$ depends on the current state $X(t)$.
As a result, switching diffusion models become increasingly popular.

Assuming that the systems are running in continuous times,
stemming from stochastic differential equations based models and random discrete events,
switching diffusions came into being, which feature in continuous states and discrete events
(discrete states)
 coexist and interact.
  Earlier considerations required the discrete states being a continuous-time Markov chain that is independent of the driving random disturbances
   (the Brownian motion) for the continuous states.
Subsequent studies put emphases on the case the discrete states (discrete random events) depend on the continuous states \cite{YinZ10}, which substantially enlarges the systems can be treated but makes the analysis much more difficult to handle.

Until very recently, most of
the works treat
 $\alpha(t)$ as a process taking values in a finite set.
In very recent considerations,   $\alpha(t)$ is allowed to take values in
 a countable state space, with the proviso
 the systems being memoryless.
To be able to treat more realistic models and to broaden the applicability,
we undertook the task of
 investigating the dynamics of $(X(t),\alpha(t))$ in which $\alpha(t)$ has a countable state space
and its switching intensities depend on the history of the continuous component  $X(t)$.
As a first attempt, motivated by, for example, applications in queueing systems and control systems among others, a systematic study of
such switching diffusions was initiated in \cite{DY1, DY2}.
In \cite{DY1}, we gave precise formulation of the process  $(X(t),\alpha(t))$ and established
the existence and uniqueness of solutions together with such
 properties as Markov-Feller property
 and Feller property of  function-valued stochastic processes associated with our processes
under suitable conditions.
In \cite{DY2}, general conditions for recurrence and ergodicity of
switching diffusion processes with past-dependent switching having a countable state space
are given using appropriate Lyapunov functions.
In practice, it is often difficult to find a suitable Lyapunov function.
With such motivations, this paper aims to provide easily verifiable conditions for positive recurrence
when the systems are linearizable.
As coined by Wonham, recurrence is also known as weak stability.
This paper also
examines weak-stabilization problems, which
 is of practical importance because
in control and optimization, dealing with long-term average
cost problems, we often need to replace the
instantaneous transition
probability measures by invariant measures whenever possible. To make the replacement,
we need the existence of the corresponding invariant
measures and the convergence of transition probability measure to the invariant measures.
When the systems encountered are not positive recurrent (also known as weakly stable),
 it will be necessary
to design controls so that the controlled switching diffusions
are weakly stable or positive recurrent, which guarantees the
existence of stationary measures.

The rest of the paper is organized as follows. The formulation of switching diffusions with past-dependent switching and
countably many possible switching locations is given in Section \ref{sec:2}.
Section \ref{sec:3} concentrates on recurrence and ergodicity, and
provides certain sufficient conditions for positive recurrence and ergodicity of
the related switching diffusions.
In Section \ref{sec:5}, we concentrate on
weak stabilization of switching diffusions by a linear feedback controls.
To demonstrate our results, we provide a few examples
in Section \ref{sec:6}.
 A short summary is given in Section \ref{sec:rem}.
Finally, for completeness,
an appendix is provided at the end of the paper 
 to cover some technical complements.

\section{Formulation and Examples}\label{sec:2}
Let $r$ be a fixed positive number.
Denote by $\C([a,b],\R^\nz)$ the set of  $\R^\nz$-valued continuous functions defined on $[a, b]$.
In what follows,
we mainly work with $\C([-r,0],\R^\nz)$, and simply denote it by $\C:=\C([-r,0],\R^\nz)$.
For $\phi\in\C$, we use the sup norm $\|\phi\|=\sup\{|\phi(t)|: t\in[-r,0]\}$.
For $t\geq0$, we denote  by $y_t$
the
 segment function or memory segment function $y_t =\{ y(t+s): -r \le s \le 0\}$.
For $x\in\R^\nz$, denote by $|x|$ the Euclidean norm of $x$.
Let $(\Omega,\F,\{\F_t\}_{t\geq0},\PP)$ be a complete filtered probability space with the filtration $\{\F_t\}_{t\geq 0}$ satisfying the usual condition,  i.e., it is increasing and right continuous while $\F_0$ contains all $\PP$-null sets.
Let $W(t)$ be an $\F_t$-adapted and $\R^d$-valued Brownian motion.
Suppose $b(\cdot,\cdot): \R^\nz\times\N\to\R^\nz$ and $\sigma(\cdot,\cdot): \R^\nz\times\N\to\R^{\nz\times d}$, where $\N= {\mathbb N} \setminus \{0\}=\{1,2,\dots\}$,  the set of positive integers.
Consider the two-component process $(X(t),\alpha(t))$, where
 $\alpha(t)$ is a pure jump process taking value in  $\N$, and $X(t)$ satisfies
\begin{equation}\label{eq:sde} dX(t)=b(X(t), \alpha(t))dt+\sigma(X(t),\alpha(t))dW(t).\end{equation}
We assume that the jump intensity of $\alpha(t)$ depends on the trajectory of $X(t)$ in the interval $[t-r,t]$,
that is, there are functions $q_{ij}(\cdot):\C\to\R$ for $i,j\in\N$ satisfying
$q_{ij}(\phi)\geq0$, for all $i\ne j$ and $\sum_{j=1}^\infty q_{ij}(\phi)=
0$. Denote
$q_{i}(\phi) = \sum^\infty_{j\not = i} q_{ij}(\phi)$ for each $i\in \N$ and each $\phi\in\C$ and $Q(\phi)=(q_{ij}(\phi))$.
For implicity, we assume that $q_i(\phi)$ is uniformly bounded (for treating more general classes of functions, cf. \cite{DY1}) and
\begin{equation}\label{eq:tran}\begin{array}{ll}
&\disp \PP\{\alpha(t+\Delta)=j|\alpha(t)=i, X_s,\alpha(s), s\leq t\}
\quad=q_{ij}(X_t)\Delta+o(\Delta) \text{ if } i\ne j \
\hbox{ and }\\
&\disp \PP\{\alpha(t+\Delta)=i|\alpha(t)=i, X_s,\alpha(s), s\leq t\}
\quad=1-q_{i}(X_t)\Delta+o(\Delta).\end{array}\end{equation}
Note that in \eqref{eq:tran}, in contrast to \cite{YinZ10},
 in lieu of $Q(X(t))$, $Q(X_t))$ (i.e., the segment process
 $X_t$) is used.
A strong solution to  \eqref{eq:sde} and \eqref{eq:tran} on $[0,T]$ with initial data $(\xi,i_0)$
being
$\C\times\N$-valued and $\F_0$-measurable random variable,
is an $\F_t$-adapted process $(X(t),\alpha(t))$ such that
\begin{itemize}
  \item $X(t)$ is continuous and $\alpha(t)$ is cadlag (right continuous with left limits) with probability 1 (w.p.1).
  \item $X(t)=\xi(t)$ for $t\in[-r,0]$ and  $\alpha(0)=i_0$
  \item $(X(t),\alpha(t))$ satisfies \eqref{eq:sde} and \eqref{eq:tran} for all $t\in[0,T]$ w.p.1.
\end{itemize}

Alternatively, the switching diffusion above may also be given
as follows.
Define $h:\C\times\N\times\R\mapsto\R$ by
$h(\phi, i, z)=\sum_{j=1, j\ne i}^\infty(j-i)\1_{\{z\in\Delta_{ij}(\phi)\}}.$
The process $\alpha(t)$ can be defined as the solution to
$$d\alpha(t)=\int_{\R}h(X_t,\alpha(t-), z)\p(dt, dz),$$	
where $a(t-)=\lim\limits_{s\to t^-}\alpha(s)$ and $\p(dt, dz)$ is a Poisson random measure with intensity $dt\times\m(dz)$ and $\m$ is the Lebesgue measure on $\R$ such that
$\p(dt, dz)$ is
independent of the Brownian motion $W(\cdot)$.
The pair $(X(t),\alpha(t))$ is therefore a solution to
\begin{equation}\label{e2.3}
\begin{cases}
dX(t)=b(X(t), \alpha(t))dt+\sigma(X(t),\alpha(t))dW(t) \\
d\alpha(t)=\disp\int_{\R}h(X_t,\alpha(t-), z)\p(dt, dz).\end{cases}
\end{equation}
For subsequent use, define
\begin{equation}\label{A-def} A(x,i)=\sigma(x,i) \sigma^{\sf T}(x,i).\end{equation}

To get some insight, we consider a couple of examples below.
One is a fluid model from queueing systems, which begins with a switched ordinary differential equations, where the probability distribution of the switching process depends on the past history. The other
example stems from applications in  ecological systems.

\begin{exm}\label{exm2} {\rm
In this example, we consider an extension of the Markov-modulated-rate fluid models treated in \cite{YZZ13}.
Stemming from queueing systems, this example is simple in that it is even without the Brownian motion part, but it explains the modeling view point of the past depend switching with a countable state space.
Consider the fluid buffer model with an infinite
capacity. Let $X(t)$ be the amount of fluids in the buffer at
time $t$,  known as  buffer content or  buffer
level. The fluids enter and leave the buffer at random rates. The
input and output of fluids  are modulated by a switching process
$\alpha(t)$ with state space $\N=\{1,2,\ldots\}$,  known
as a stochastic external environment. Using $\alpha(t)$ to
determine the input and output rates, we introduce a
 drift function $f(\cdot)$ (see Kulkarni \cite{kulkarni}):
$ f\cd: \N  \mapsto  (-\infty, \infty).$
Different from \cite{YZZ13}, we do not assume $\alpha(t)$ to be a Markov chain,
but rather, we assume that the transition rates satisfy \eqref{eq:tran}.
That is, the transition rates depend on the history of the process $X(t)$.

Note that formally, the net rate that is
the difference of the input  and the output rates at time
$t$ is given by $f(\alpha(t))$. The dynamics of the
buffer content $\{X(t): t\geq 0\}$ can be described
by the following differential equation:
\beq{dyna-equa}
\barray
\disp\frac{d}{dt}X(t)\ad =\left\{
\begin{array}{ll}
f(\alpha(t)), &\mbox{if $X(t)>0$}\\
(f(\alpha(t)))^+, &\mbox{if $X(t)=0$},
\end{array}
\right.\\
\ad = f(\alpha(t))\1_{\{X(t)>0\}}
+(f(\alpha(t)))^+\1_{\{X(t)=0\}},\earray
\eeq
where $x^+=\max\{x,0\}$.
Note that $X(t)$ can be rewritten as
\beq{cont-0}
X(t)=X(0)+Y(t)-\left(\inf_{0\leq s \leq
t}\left\{Y(s)+X(0)\right\}\right)\wedge 0 \eeq
with
$$Y(s)=\int^s_0f(\alpha(v))dv,$$
where $a \wedge b=\min (a,b)$ for two real numbers $a$ and $b$,
and $$-\left(\inf_{0\leq s \leq t}\{Y(s)+X(0)\}\right)\wedge 0$$ measures
the amount of potential output lost up to time $t$ due to
the emptiness of the buffer.
Many of the current interests are concerned with the fluid model above such as long-run average control problems or stability of the systems.
}
\end{exm}

\begin{exm}\label{exm:1a} {\rm
Consider the evolution of a predator-prey model
in which the predator is macro and the prey is micro.
Denote by $X(t)$ and $\alpha(t)$ the density of the predator species
and the number of the prey at time $t$ respectively.
In \cite{MC},
the dynamics of the predator is given by a differential equation
\begin{equation}\label{ex3-e1}
d X(t)=X(t)\Big(\rho B\alpha(t)-D-CX(t)\Big)dt
\end{equation}
and
the number of the prey evolves according to a birth and death process
with switching rates given by
\begin{equation}\label{ex3-e2}
\begin{aligned}
\PP&\left\{\alpha(t+s)=j\Big|\alpha(t)=n,X(t)\right\}\\
&
=\begin{cases}
\beta ns+O(s)&\,\text{ if } j=n+1,n\geq 1\\
n(\delta+cn+BX(t))s+O(s)&\,\text{ if } j=n-1,n\geq 1\\
1-n(\beta+\delta+cn+BX(t))s+O(s)&\,\text{ if } n=j,n\geq 2\\
1-\beta s+O(s)&\,\text{ if } j=n=1\\
O(s)&\,\text{ otherwise. }
\end{cases}
\end{aligned}
\end{equation}
In \eqref{ex3-e2},
$\beta$ and $\delta$ represent
the birth rate and the death rate of the predator,
$D$ is the  death rate of the prey,
$C$ and $c$ are the intraspecific competition rates
of the prey and predator respectively.
$B$ is the loss rate of the prey due to the predation,
while $\rho$ is the intake rate of the predator.

Suppose that
the  dynamics of $X(t)$ is subject to  environmental noise
describled by a Brownian motion.
Since
the life cycle of a micro species is usually very short,
so it is reasonable to assume that the dynamical equation
 of $X(t)$ is past-independent.

On the other hand, for the micro species $\alpha(t)$, the reproduction process of $\alpha(t)$ is assumed
 to be non-instantaneous.
More precisely, suppose that the reproduction depends on the period of time from egg formation to hatching, say $r$.
then we can assume that
the switching rate of $\alpha(t)$
depends on a history of $X(\cdot)$ from $t-r$ to $t$
rather than the current state $X(t)$.
Thus, we may replace $X(t)$ in \eqref{ex3-e2}
by a functional $\int_{-r}^0X(t+u)\mu(du)$
where $\mu$ is some measure on $[-r,0]$.

With these assumptions,
the model \eqref{ex3-e1} and \eqref{ex3-e2}
will become
\begin{equation}\label{ex3-e3}
\begin{cases}
&d X(t)=X(t)\Big(\rho B\alpha(t)-D-CX(t)\Big)dt+\sigma X(t)dW(t)\\
&\PP\left\{\alpha(t+s)=j\big|\alpha(t)=n,X(s), s\leq t\right\}=q_{ij}(X_t)s+O(s),\, n\ne j\\
&\PP\left\{\alpha(t+s)=n\big|\alpha(t)=n,X(s), s\leq t\right\}=1-q_{n}(X_t)s+O(s),\, n\geq 1.
\end{cases}
\end{equation}
where
$$
q_{ij}(\phi)
=\begin{cases}
\beta ns+O(s)&\,\text{ if } j=n+1,n\geq 1\\
n\left(\delta+cn+B\int_{-r}^0\phi(u)\mu(du)\right)s+O(s)&\,\text{ if } j=n-1,n\geq 1\\
0&\,\text{ if } j\notin \{n-1, n, n+1\},
\end{cases}
$$
and $q_i(\phi)=\sum_{j\ne i}q_{ij}(\phi)$.
We want to answer the question: under which conditions the species will be permanent forever or  they will extinct at some instance? Whether or not there is an invariant measure associated with the system under consideration. These questions are related to the stability and ergodicity of the
corresponding stochastic systems.
}
\end{exm}

\section{Recurrence}\label{sec:3}

In  reference \cite{DY2}, we examined the issue of recurrence and obtained sufficient conditions ensuring positive recurrence. We also examined the ergodicity
of the underlying processes.
In this paper, as a point of departure, we aim to obtain more easily verifiable conditions. We are trying to derive conditions that focusing on the coefficients of the drifts and diffusion matrices. As pointed out by \cite{Wonham}, recurrence may be
called weak stability. It is well known that stability normally is referred to properties of solutions at certain stationary points.  Likewise, recurrence can also be regarded as such. However, the neighborhood is no longer about a finite point, but rather the neighborhood of infinity.
In stability analysis, we often wish to study something called near linear systems,
which are systems locally like a linear one with high order terms involved. Here, we have similar things. We wish to linearize the systems about the point of ``$\infty$''. Then we wish to see if the linearized systems are weakly stable (recurrent), is it true the nonlinear systems are as well.
Before proceeding further, let us recall the definition of recurrence.

\begin{defn}{\rm
The process $\{(X_t,\alpha(t)): t\geq0\}$ is said to be recurrent
if for any bounded set $\D\subset \C$ and a finite set $N\subset\N$,
we have
if $$\PP_{\phi,i}\{(X_t,\alpha(t))\in\D\times N\,\text{ for some } t\geq 0\}=1.$$
It is said to be  positive recurrent if
$$\E_{\phi,i}\tau_{D,N}<\infty \ \hbox{
 for any }\ (\phi,i)\in\C\times\N,$$
 where
 $\tau_{D,N}=\inf\{t\geq0: X_t\in D, \alpha(t)\in N\}$.
}\end{defn}

To obtain the recurrence,
we need an assumption that guarantees
the irreducibility of the process $(X_t,\alpha(t))$.
Thus, we impose the following conditions.

\begin{asp}\label{asp2.1}{\rm
The following conditions hold.
\begin{enumerate}
\item For each $i\in\N$, $H>0$, there is a positive constant $ L_{i,H}$ such that
$$|b(x,i)-b(y,i)|+|\sigma(x,i)-\sigma(y,i)|\leq  L_{i,H}|x-y|$$
if $x,y\in\R^\nz$ and $|x|,|y|\leq H$.
\item  $q_{ij}(\phi)$ is continuous in $\phi\in\C$ for each $(i,j)\in\N^2$.
\item
For any $H>0$,
      $$M_H:=\sup_{\phi\in\C, \|\phi\|\leq H, i\in\N}\{q_{i}(\phi)\}<\infty.$$
\end{enumerate}
}\end{asp}

\begin{asp}\label{asp2.2}
{\rm Suppose either of the following condition satisfied.
\begin{enumerate}
\item
 For any $i\in\N,$ $A(x,i)$
is elliptic uniformly on each compact set, that is, for any $R>0$, there is a $\theta_{R,i}>0$ such that
\begin{equation}\label{ellip}
y^\top A(x,i)
y\geq \theta_{R,i}
|y|^2
\ \ \forall |x|\leq R, \
y\in\R^\nz.
\end{equation}
Moreover, for any $i, j\in\N$, there exist $i_1,\dots,i_k\in\N$ and $\phi_1,\dots,\phi_{k+1}\in\C$ such that
$q_{ii_1}(\phi_1)>0$, $q_{i_l,i_{l+1}}(\phi_{l+1})>0,l=1,\dots, k-1$, and $q_{i_k,j}(\phi_{k+1})>0$.
\item     There exists $i^*\in\N$ such that $A(x,i^*)$
is elliptic uniformly on each compact set.
Moreover,
for any $i, j\in\N$, $\phi\in C$, there exist $i_1,\dots,i_k\in\N$ such that
$q_{ii_1}(\phi)>0$, $q_{i_l,i_{l+1}}(\phi)>0,l=1,\dots, k-1$, and $q_{i_k,j}(\phi)>0$.
\end{enumerate}
}
\end{asp}

\begin{asp}\label{asp3.1}{\rm
Suppose that for $i\in\N$, there exist $b(i),\sigma_k(i)\in\R^{\nz\times \nz}$ bounded uniformly for $i\in\N$ such that
$\hat b(x,i):=b(x,i)-b(i)x$ and $\hat\sigma(x,i):=\sigma(x,i)-(\sigma_1(i)x,\dots,\sigma_d(i)x)$
satisfying
\begin{equation}\label{e3-ex2}
\lim_{x\to\infty} \sup_{i\in\N}\left\{\dfrac{|\hat b(x,i)|\vee |\hat\sigma(x,i)|}{|x|}\right\}=0.
\end{equation}
There is $\hat Q=(\hat q_{ij})_{\N\times\N}$
such that $\hat Q$ is irreducible and
\begin{equation}\label{e1-thm3.1}
\lim_{R\to\infty}\left(\sup_{i\in\N,\phi\in\D_R} \sum_{j\ne i}|q_{ij}(\phi)-\hat q_{ij}|\right)\to 0
\end{equation}
where $\D_R=\{\phi\in\C: |\phi(t)|\geq R\,\forall\,t\in[0,R]\}$
}\end{asp}

\begin{asp}\label{asp3.2} {\rm
There exists $k_0\in\N$ and a bounded non-negative sequence
$\{\eta_k: k>k_0\}$
such that
$$
\sum_{j>k_0} q_{ij}(\phi)\eta_j\leq -1\,\text{ for any }\phi\in\C.
$$}
\end{asp}

\begin{rem}{\rm
This assumption stems from a condition for strong ergodicity
of a Markov chain having a countable state space.
We will show that, under this assumption,
$\sup_{(\phi,i)\in\C\times\N}\E_{\phi,i}\zeta<\infty$
where $\zeta=\inf\{t\geq0: \alpha(t)\leq k_0\}.$
}\end{rem}

\begin{rem}{\rm
Let us briefly comment on the conditions.
\begin{itemize}

\item Assumption \ref{asp2.2} consists of two parts. The first part is a condition on the ellipticity, whereas the second is to ensure the irreducibility of the switching process. One view this condition as in continuous-time Markov chain ensuring certain positive probabilities. However, because of the past dependence, these quantities now depend on $\phi$.

\item Assumption \ref{asp3.1} essentially indicates that the drifts and diffusion matrices can be linearized in the neighborhood of infinity. The $Q(\phi)$ can also be approximated in a certain sense by a generator of a Markov chain.

\item Because the switching takes values in $\N$, its state may not be bounded,
Assumption \ref{asp3.2} requires that the switching does not act too wildly. In fact, it is ``pushed in'' in the sense that Assumption \ref{asp3.2} is satisfied.

\item It can be easily seen that under Assumption \ref{asp3.1}, there exists $\tilde L>0$ such that
$$|b(x,i)|+|\sigma(x,i)|\leq \tilde L(|x|+1)\,\text{ for any }\, (x,i)\in\R^\nz\times\N.$$
Thus, Assumptions \eqref{asp2.1} and \eqref{asp3.1} are sufficient
for the existence and uniqueness of solutions $(X(t),\alpha(t))$ to \eqref{e2.3}
due to \cite[Theom 3.5]{DY1}. Moreover, the solution process has the Markov-Feller property
according to \cite[Theorems 4.8]{DY1}.
\end{itemize}
}
\end{rem}

Throughout this paper,
we assume that
Assumptions \ref{asp2.1}, \ref{asp2.2}, \ref{asp3.1}, and \ref{asp3.2}
are satisfied.
Then, we have some auxiliary results whose proofs are relegate to the appendix.

\begin{lm}\label{lm3.1}
There is a $K_1>0$ such that
$$\E_{\phi,i} \sup_{t\in[0,T]}\{|X(t)|^2\}\leq K_1(1+|\phi(0)|)e^{K_1T},$$
for any $T>0$, $(\phi,i)\in\C\times\N$.
\end{lm}

\begin{lm}\label{lm3.2}
For any $K_2>0$,  $\eps>0$, and $T>0$,
there exists $K_3>0$ such that
$$\E_{\phi,i}\{X(t)\geq K_2\, \text{ for }\, t\in[0,T]\}>1-\eps$$
given that $\phi(0)\geq K_3$.
\end{lm}

\begin{lm}\label{lm3.3}
$$\sup_{(\phi,i)\in\C\times\N}\E_{\phi,i}\zeta<\infty$$
where $\zeta=\inf\{t\geq0: \alpha(t)\leq k_0\}.$
Moreover,
$\hat\alpha(t)$ is a strong ergodic process, that is,
$$\lim_{t\to\infty}\sup_{i\in\N}\bigg\{\sum_{j\in\N}|\hat p_{ij}(t)-\nu_j|\bigg\}=0.$$
where $\hat p_{ij}(t), i,j\in\N$ is the transition probability of $\hat\alpha(t)$
and $\bnu=(\nu_1,\nu_2,\dots)$ is an invariant probability measure of $\hat\alpha(t)$.
\end{lm}

\begin{lm}\label{lm3.4}
For any $T>0$ and a bounded function $f:\N\mapsto\R$, we have
\begin{equation}\label{e3-a2}
\lim_{R\to\infty} \sup_{\phi\in\D_R,i\in\N,t\in[0,T]}\Big\{\Big|\E_{\phi,i}f(\alpha(t))-\E_{i}f(\hat\alpha(t)\Big|\Big\}=0.
\end{equation}

\end{lm}

\begin{lm}\label{lm-a3}
Let $Y$ be a random variable, $\theta_0>0$ a constant, and suppose $$\E \exp(\theta_0 Y)+\E \exp(-\theta_0 Y)\leq K_1.$$
Then the log-Laplace transform
$\phi(\theta)=\ln\E\exp(\theta Y)$
is twice differentiable on $[0,0.5\theta_0]$ and
$$\dfrac{d\phi}{d\theta}(0)= \E Y,\quad\text{ and }\,0\leq \dfrac{d^2\phi}{d\theta^2}(\theta)\leq K_2\,, \theta\in[0,0.5\theta_0]$$
 for some $K_2>0$.
As a result of Taylor's expansion,
we have
$$
\phi(\theta)\leq\theta\E Y +\theta^2 K_2,\,\,\text{ for }\theta\in[0,0.5\theta_0).
$$
\end{lm}

Now, we state the main result of the paper.
For a matrix $A\in\R^{\nz\times \nz}$,
denote by $\Lambda_A$ and $\lambda_A$,
the maximum and minimum eigenvalues of $\frac{A+A^\top}2$
respectively.
Let $\rho_A=\inf_{|x|=1}\{|x^\top A x|\}$.

\begin{thm}\label{thm:main}
Let Assumptions \ref{asp2.1}, \ref{asp2.2}, \ref{asp3.1} and \ref{asp3.2}
be satisfied.
Assume further that
$$\sum_{i\in\N}\nu_i\left(\Lambda_{b(i)}+\dfrac{\Lambda_{a(i)}}2-\sum_{j=1}^d\rho_{\sigma_j(i)}^2\right)<0$$
where $a(i)=\sum_{j=1}^d \sigma_j^\top(i)\sigma_j$.
Then the process $(X_t,\alpha(t))$ is positive recurrent.
Moreover, the process has a unique invariant probability measure $\mu^*$  such that for any $(\phi, i)\in\C\times\N$
\begin{equation}\label{e0-thm2.4}
\lim_{t\to\infty}\|P(t,(\phi,i),\cdot)-\mu^*\|_{TV}=0.
\end{equation}
where $P(t,(\phi,i),\cdot)$ is the transition probability of
$(X_t,\alpha(t))$ and $\|\cdot\|_{TV}$ is the total variation norm.
\end{thm}

\begin{rem}{\rm
This result is similar to \cite[Theorem 5.1]{ZY2},
which is based on the Fredholm alternative.
However, in our setting, $\alpha(t)$ takes value in a countable state space.
The Fredholm alternative is not applicable.
Thus, we need different approach to obtain the desired result.
}\end{rem}

\begin{proof}  [Proof of Theorem {\rm\ref{thm:main}}]
Let $\eps_0>0$ sufficiently small that
$$-\lambda:=\sum_{i\in\N}\nu_i\left(\eps_0+\Lambda_{b(i)}+\dfrac{\Lambda_{a(i)}}2-\sum_{j=1}^d\rho_{\sigma_j(i)}^2\right)<0$$
Denote by $\LL_i$ the generator of the diffusion when the discrete component is in state $i$, that is,
$$\LL_iV(x)=\nabla V(x)b(x,i)+\dfrac12\trace\Big(\nabla^2 V(x)A(x,i)\Big).$$
where $\nabla f(x,i)=(f_1(x,i),\dots,f_n(x,i))\in \rr^{1\times \nz}$
and $\nabla^2 f(x,i)=(f_{ij}(x,i))_{\nz\times \nz}$ are the gradient and Hessian of $f(x,i)$ with respect to $x$, respectively,
 with \bea \ad f_k(x,i) =(\partial /\partial x_k) f(x,i),\
f_{kl} (x, i) = (\partial^2/\partial x_k \partial x_l) f(x,i),\
\hbox{ and }\\
\ad A(x,i)=(a_{ij}(x,i))_{\nz\times \nz}=\sigma(x,i)\sigma^\top(x,i),\eea
where $z^\top$ denotes the transpose of $z$.
By direct computation and \eqref{e3-ex2}, we have
\begin{equation}\label{e1-thm3.1}
\begin{aligned}
\LL_i \ln|x|=&
\dfrac {x^\top b(x,i)}{|x|^2}+\dfrac12\dfrac{\trace\big(\sigma^\top(x,i)(i)\sigma(x,i)\big)}{|x|^2}-\dfrac{\trace\big(\sigma^\top(x,i)xx^\top\sigma(x,i)\big)}{|x|^4}\\
=&\Xi(x,i)+\dfrac{x^\top b(i)x}{|x|^2}+\sum_{j=1}^d\left(\dfrac12\dfrac{x^\top\sigma_j^\top(i)\sigma_j(i)x}{|x|^2}-\dfrac{(x^\top\sigma_j^\top x)^2}{|x|^4}\right)\\
\leq& \Xi(x,i)+\Lambda_{b(i)}+\dfrac{\Lambda_{a(i)}}2-\sum_{j=1}^d\rho_{\sigma_j(i)}^2
\end{aligned}
\end{equation}
where
$\Xi(x,i)$ satisfies that
$\lim_{x\to\infty} \sup_{i\in\N}\Xi(x,i)=0$.
Let $H>0$ such that
$\Xi(x,i)\leq \eps_0$ for any $|x|\geq H$.
Let $\tau=\inf\{t\geq0: |X(t)|<H\}$.
By It\^o's formula and \eqref{e1-thm3.1},
$$
\begin{aligned}
\ln |X(t\wedge\tau_{H})|=&\ln |X(0)| +\int_0^{t\wedge\tau_{H}}\LL\ln|X(s)| ds\\
&+\int_0^{t\wedge\tau_{H}} \dfrac{X^\top(s)\sigma(X(s),\alpha(x))}{|X(s)|^2}dW(s)ds\\
\leq&\ln |X(0)|+G(t)
\end{aligned}
$$
where \bea \ad G(t)=\int_0^{t\wedge\tau_{H}}c(\alpha(s)) ds+\int_0^{t\wedge\tau_{H}} \wdt\sigma(X(s),\alpha(x))dW(s)ds,\\
\ad
c(i)=\eps_0+\Lambda_{b(i)}+\dfrac12\Lambda_{a(i)}
-\sum_{j=1}^d\rho_{\sigma_j(i)}^2 \ \hbox{ and }\\ \ad \wdt\sigma(x,i)=\dfrac{x^\top\sigma(x,i)}{|x|^2}.\eea
In view of \eqref{e3-ex2} and the boundedness of $\sigma(i)$,
we have
$$M_\sigma=\sup_{(x,i)\in\R^\nz\times\N}\{|\wdt\sigma(x,i)|\}<\infty.$$
By It\^o's formula,
$$\begin{aligned}
e^{\theta G(t)}=& 1+\int_0^{t\wedge\tau_H}e^{\theta G(s)}\left[\theta c(\alpha(s))+\dfrac{\theta^2}2\wdt\sigma(X(s),\alpha(s))\right]ds\\
&+
\theta\int_0^{t\wedge\tau_H}e^{\theta G(s)}\wdt\sigma(X(s),\alpha(s))dW(s),
\end{aligned}
$$
which leads to
$$
\begin{aligned}
\E_{\phi,i}e^{\theta G(t)}
=&1+\E_{x,i}\int_0^{\tau_H\wedge t}e^{\theta G(s)}\left[\theta c(\alpha(s))+\dfrac{\theta^2}2\wdt\sigma(X(s),\alpha(s))\right]ds\\
\leq&1+ [\bar c+M_\sigma]\E_{\phi,i}\int_0^{\tau_H\wedge t}e^{\theta G(s)}ds\\
\leq&1+[\bar c+M_\sigma]\int_0^{t}\E_{\phi,i}e^{\theta G(s)}ds.
\end{aligned}
$$
where $\bar c=\sup_{i\in\N}\{|c(i)|\}<\infty$.
In view of Gronwall's inequality, for any $t\geq0$ and $(\phi,i)\in \D_h\times\N$, we have
\begin{equation}\label{e12-thm3.2}
\E_{\phi,i}e^{\theta G(t)}\leq e^{\theta[\bar c+M_g]t},\, \theta\in[-1,1].
\end{equation}
On the other hand,
we have
\begin{equation}\label{e7-thm3.2}
\begin{aligned}
\E_{\phi,i}G(t)&\leq\E_{x,i}\int_0^{\tau_H\wedge t}c(\alpha(s))ds\\
&\leq\E_{\phi,i}\int_0^{t}c(\alpha(s))ds-\E_{\phi,i}\int_{\tau_H\wedge t}^tc(\alpha(s))ds\\
&\leq\E_{\phi,i}\int_0^{t}c(\alpha(s))ds+t\bar c\PP_{\phi,i}\{\tau_H<t\}.
\end{aligned}
\end{equation}
Let $\wdt\eps>0$ such that
\begin{equation}\label{e3-thm3.1}
\bar c\wdt\eps-(1-\wdt\eps)0.8\lambda \leq -0.7
\end{equation}
Because of the uniform ergodicity of $\hat\alpha(t)$,
there exists a $T>10\bar c r\lambda^{-1}+r$ such that
\begin{equation}\label{e8-thm3.2}
\E_{i}\int_0^{T-r} c(\hat\alpha(s))ds\leq -0.9\lambda T,\text{ for any } i\in\N.
\end{equation}
By Lemma \ref{lm3.4},
there exists an $H_1>H$ such that
\begin{equation}
\E_{\phi,i}\int_0^{T-r}c(\alpha(s))ds\leq -0.8\lambda T, \|\phi\|>H_1, i\in\N,
\end{equation}
where $T_2=(m_0+1)T$.
In light of Lemma \ref{lm3.2}, there exists an $H_2>H_1$ such that
\begin{equation}\label{e10-thm3.2}
\bar c\PP_{\phi,i}\{\tau_H<T+r\}\leq 0.1\lambda\,\text{ provided } \|\phi\|\geq H_2, i\in\N.
\end{equation}
and
$$\PP_{\phi,i}\{ X(t)\geq H_1, t\in[0,r]\}>1-\wdt\eps \
\hbox{ if }\ |\phi(0)|\geq H_2.$$
Using these and  the Markov property of $(X_t,\alpha(t))$,
we have
\begin{equation}\label{e9-thm3.2}
\begin{aligned}
\E_{\phi,i}\int_0^{T}c(\alpha(s))ds
=&\E_{\phi,i}\int_0^{r}c(\alpha(s))ds+\E_{\phi,i}\left(\1_{\{|X_t|\geq H_1, t\in[0,r]\}}\int_r^{T}c(\alpha(s))ds\right)\\
&+\E_{\phi,i}\left(\1_{\{|X_t|< H_1\,\text{ for some }t\in[0,r]\}}\int_r^{T}c(\alpha(s))ds\right)\\
\leq&\bar c (r+\wdt\eps T)-(1-\wdt\eps)0.8\lambda T\\
\leq& -0.6\lambda T.
\end{aligned}
\end{equation}
where the last inequality follows from $T>10\bar c r\lambda^{-1} $
and \eqref{e3-thm3.1}.
Applying \eqref{e9-thm3.2} and \eqref{e10-thm3.2} to \eqref{e7-thm3.2},
we obtain
\begin{equation}\label{e11-thm3.2}
\E_{x,i}G(T)\leq -0.5\lambda T \,\text{ if }\, |\phi(0)|\geq H_2.
\end{equation}
By Lemma \ref{lm-a3}, it follows from \eqref{e12-thm3.2} and \eqref{e11-thm3.2} that for $\theta\in[0,0.5], 0<|x|<h_2, i\leq k_0,$ we have
\begin{equation}\label{e13-thm3.2}
\begin{aligned}
\ln \E_{x,i}e^{\theta G(T)}\leq& \theta\E_{x,i}G(T)+\theta^2 \wdt K
\leq-0.5\theta \lambda+\theta^2 \wdt K
\end{aligned}
\end{equation}
for some $\wdt K>0$ depending on $T, \bar c$, and $M_\sigma$.
Let $\theta\in(0,0.5]$ such that
$
\theta\wdt  K<\frac{\lambda_1 T}4$.
We have
$$
\ln \E_{\phi,i}e^{\theta H(T)}\leq -\dfrac{\theta\lambda_1 T}4\,\text{ for }|\phi(0)|\geq H_2,
$$
or equivalently,
\begin{equation}\label{e15-thm3.2}
\E_{x,i}e^{\theta H(T)}\leq e^{-\frac{\theta\lambda_1 T}4}\,\text{ for }|\phi(0)|\geq H_2.
\end{equation}
Exponentiating both sides of the inequality
$$\ln|X(T\wedge\tau_H|\leq \ln |X(0)|+G(T),$$
we have from \eqref{e15-thm3.2} that
\begin{equation}\label{e16-thm3.2}
\E_{\phi,i} |X(T\wedge\tau_H)|^\theta\leq |\phi(0)|^\theta e^{-\frac{\theta\lambda_1 T}4}\,\text{ for }|\phi(0)|\geq H_2,
\end{equation}
where $\rho=e^{-\frac{\theta\lambda_1 T}4}<1$.
Let $\xi:=\inf\{k\in\mathbb{N}: X(kT\wedge \tau_H)\leq H_2\}$.
Applying the Markov property of $(X_t,\alpha(t))$ to \eqref{e16-thm3.2} yields
$$
\E_{\phi,i}\1_{\{\xi\geq k+1\}}|X((kT+T))|^\theta
\leq \rho \E_{\phi,i}\1_{\{\xi\geq k\}}|X(kT)|^\theta
.$$
Using this recursively, we have
$$
\PP_{\phi,i}\{\xi\geq k+1\}\leq \dfrac{\E_{\phi,i}\1_{\{\xi\geq k+1\}}|X((kT+T))|^\theta}{H_2^\theta}
\leq \rho^k\dfrac{|\phi(0)|^\theta}{H_2^\theta}.$$
Thus,
\begin{equation}\label{e5-thm3.1}
\E_{\phi,i}\tau_{H_2}\leq T\E_{\phi,i}\xi
\leq\dfrac{|\phi(0)|^\theta}{H_2^\theta}\sum_{k=0}^\infty (k+1)\rho^k
=:C_1 |\phi(0)|^\theta.
\end{equation}
In view of Lemma \ref{lm3.3},
there exists a $T_3>r$ such that
\begin{equation}\label{e6-thm3.1}
\PP_{\phi,i}\Big\{\alpha(t)\leq k_0\,\text{ for some }\, t\in[r,T_3]\Big\}>\frac34
\text{ if }|\phi(0)|\leq H_2.
\end{equation}
On the other hand, By Lemma \ref{lm3.1}, there exists an $H_3>0$ satisfying
\begin{equation}\label{e7-thm3.1}
\PP_{\phi,i}\Big\{|X(t)|\leq H_3, t\in[0,T_3]\Big\}>\frac34
\text{ if }|\phi(0)|\leq H_2.
\end{equation}
Combining \eqref{e6-thm3.1} and \eqref{e7-thm3.1} yields
\begin{equation}\label{e8-thm3.1}
\PP_{\phi,i}\Big\{\|X_t\|\leq H_3\text{ and } \alpha(t)\leq k_0 \text{ for some } t\in[r,T_3]\Big\}>\frac12
\text{ if }|\phi(0)|\leq H_2.
\end{equation}
Define stopping times
$$\tau^{(1)}=\tau_{H_2}=\inf\{t\geq 0: |X(t)|\leq H_2\},$$
$$\tau^{(k)}=\inf\left\{t\geq T_3+\tau^{(k-1)}: |X(t)|\leq H_2\right\}\ \text{ for } \ k\geq 2,$$
and events
$$B_k=\left\{\|X_t\|\leq H_3\text{ and } \alpha(t)\leq k_1 \text{ for some } t\in\left[\tau^{(k)},\tau^{(k)}+T_3\right]\right\}, k\geq0.$$
Let $\sigma^{(1)}=\tau^{(1)}, \sigma^{(k)}=\tau^{(k)}-\tau^{(k-1)}, k\geq 2$.
Note that
if $B_k$ occurs, then $\vartheta\leq \tau^{(k)}+T_3$.
By the strong Markov property of $(X_t,\alpha(t))$, \eqref{e5-thm3.1}
$$
\begin{aligned}
\PP_{\phi,i}\{\cap_{k=1}^{n+1}B^c_k\}
=&\E_{\phi,i}\left[\1_{\{\cap_{k=1}^nB^c_k\}}\E\left[B^c_{n+1}\big|\F_{\tau^{(n+1)}}\right]\right]\\
\leq&\dfrac12\E_{\phi,i}\left[\1_{\{\cap_{k=1}^nB^c_k\}}\right]\\
=&\dfrac12\PP_{\phi,i}\{\cap_{k=1}^{n}B^c_k\}.
\end{aligned}
$$
By induction, we have
$$
\PP_{\phi,i}\{\cap_{k=n_0}^{n+1}B^c_k\}\leq 2^{-n+n_0},
$$
which leads to
$$
\PP_{\phi,i}\{\cap_{k=n_0}^{\infty}B^c_k\}=0.
$$
As a result,
\begin{equation}\label{e13-thm3.1}
\sum_{n=n_0}^\infty\PP_{\phi,i}\{B_n\cap_{k=n_0}^{n-1}B^c_k\}
=1,
\end{equation}
and
\begin{equation}\label{e14-thm3.1}
\begin{aligned}
\E_{\phi,i}\vartheta
=&\sum_{n=1}^\infty\E_{\phi,i}\left[\vartheta\1_{\{\B_n\cap_{k=1}^{n-1}B^c_k\}}\right]\\
\leq& \sum_{n=1}^\infty\E_{\phi,i}\left[[\tau_{H_2}^{n}+T_3]\1_{\{\B_n\cap_{k=1}^{n-1}B^c_k\}}\right]\\
\leq& T_3+\sum_{n=1}^\infty\E_{\phi,i}\left[\sum_{l=1}^n\sigma_{H_2}^{(l)}\1_{\{\B_n\cap_{k=1}^{n-1}B^c_k\}}\right]\\
=&T_3+\sum_{l=1}^\infty\E_{\phi,i}\left[\sigma_{H_2}^{(l)}\sum_{n=l}^\infty\1_{\{\B_n\cap_{k=1}^{n-1}B^c_k\}}\right]\\
=&T_3+\sum_{l=1}^\infty\E_{\phi,i}\left[\sigma_{H_2}^{(l)}\1_{\{\cap_{k=1}^{l-1}B^c_k\}}\right]\,\text{ (due to \eqref{e13-thm3.1})}.
\end{aligned}
\end{equation}
By the strong Markov property
and
$$
\E_{\phi,i}\left[|X(\tau^{(l-1)}+T_3)|^\theta\big|\F_{\tau^{(l-1)}}\right]
\leq K_1(1+H_2)e^{K_1T_3}.
$$
Using this, \eqref{e5-thm3.1} and the strong Markov property again,
we have
$$
\E_{\phi,i}\left[\tau^{(l)}-\tau^{(l-1)}-T_3)\big|\F_{\tau^{(l-1)}}\right]
\leq C_1K_1(1+H_2)e^{K_1T_3}.
$$
Thus,
\begin{equation}\label{e15-thm3.1}
\E_{\phi,i}\left[\sigma^{(l)}\big|\F_{\tau^{(l-1)}}\right]\leq M_2:=C_1K_1(1+H_2)e^{K_1T_3}.
\end{equation}
Applying the strong Markov property and \eqref{e15-thm3.1}, we have
\begin{equation}\label{e16-thm3.1}
\begin{aligned}
\E_{\phi,i}\left[\sigma^{(l)}\1_{\{\cap_{k=1}^{l-1}B^c_k\}}\right]
=&\E_{\phi,i}\left[\1_{\{\cap_{k=1}^{l-1}B^c_k\}}\E\left[\sigma^{(l)}\big|\F_{\tau^{l-1}}\right]\right]\\
\leq&M_2\E_{\phi,i}\left[\1_{\{\cap_{k=1}^{l-1}B^c_k\}}\right]\\
\leq&M_22^{-l+1},\,l\geq 2.
\end{aligned}
\end{equation}
Subsequently, it follows from
\eqref{e14-thm3.1}, \eqref{e16-thm3.1}, and \eqref{e5-thm3.1} that
$$
\begin{aligned}
\E_{\phi,i}\vartheta &\!\disp \leq T_3+\sum_{l=1}^\infty\E_{\phi,i}\left[\sigma^{(l)}
\1_{\{\cap_{k=1}^{l-1}B^c_k\}}\right]\\
&\! \disp
\leq T_3+C_1|\phi(0)|^\theta+\sum_{l=2}^\infty M_22^{-l+1}<\infty.
\end{aligned}
$$
Then, using the arguments in \cite[Theorem 3.3]{DY2},
we can show that
 $(X_t,\alpha(t))$ is positive recurrent.
 Moreover, there is a unique invariant probability measure $\mu^*$, and for any $(\phi, i)\in\C\times\N$
$$\lim_{t\to\infty}\|P(t,(\phi,i),\cdot)-\mu^*\|_{TV}=0.$$
\end{proof}

\section{Weak Stabilization}\label{sec:5}
In this section, our goal is to design suitable controls so that the regime-switching
diffusion (2.2) is positive recurrent. Consider the controlled regime-switching diffusion
\begin{equation}
\begin{aligned}
dX(t)=&\big[b(X(t),\alpha(t))+B(\alpha(t))u(X(t),\alpha(t))\big]dt+\sigma(X(t),\alpha(t))dW(t).
\end{aligned}
\end{equation}
The system is often observable only when it operates in some modes but not all.
We therefore suppose that there two disjoint
subsets $\M_1$ and $\M_2$, where for each mode $i\in \M_2$, the
process (6.2) cannot be stabilized by feedback control, but it can be stabilized for
each $i\in \M_1=\N\setminus\M_2$.
We consider feedback control of the form
$u(X(t), \alpha(t)) = −L(\alpha(t))X(t)$,
where for each $i\in \N$, $L(i)\in \R^{\nz\times \nz}$ is a constant matrix. Moreover, if $i\in \M_2$,
then $L(i) = 0$.
The controlled system becomes
\begin{equation}
\begin{aligned}
dX(t)=&\big[b(X(t),\alpha(t))-B(\alpha(t))L(\alpha(t))X(t)\big]dt\\
&+\sigma(X(t),\alpha(t))dW(t).
\end{aligned}
\end{equation}
As a corollary of Theorem \ref{thm:main}, we have
\begin{thm}
If for each $i\in\M_1$, there exists constant matrix $L(i)\in\R^{\nz\times \nz}$
such that
$$
\begin{aligned}
\sum_{i\in\M_1}&\nu_i\left(2\Lambda_{b(i)-B(i)L(i)}+\sum_{j=1}^d\Big(\Lambda_{a_j(i)}-\rho_{\sigma_j(i)}^2\Big)\right)\\
+&\sum_{i\in\M_2}\nu_i\left(2\Lambda_{b(i)}+\sum_{j=1}^d\Big(\Lambda_{a_j(i)}-\rho_{\sigma_j(i)}^2\Big)\right)
<0
\end{aligned}
$$
then the controlled regime-switching system
is weakly stabilizable $($i.e., the controlled
regime-switching diffusion is positive recurrent$)$.
\end{thm}

\section{Examples}\label{sec:6}
This section is devoted to several examples.
They are intended for demonstration purpose.

\begin{exm}\label{ex3}
{\rm It is well known that an Ornstein-Uhlenbeck (OU) process is
a non-trivial example of a process being stationary, Gaussian, and Markov.
Such processes
have been used in many applications, for example in finance.
One of the distinct properties is that it can be used to
delineate mean reversion.
In this example we
consider a switched Ornstein-Uhlenbeck process
$$
dX(t)=\theta(\alpha(t))(\mu(\alpha(t))-X(t))dt+\sigma(\alpha(t))dW(t)
$$
where $\theta(i),\mu(i),\sigma(i)$ are bounded,
$X(t)$ is real-valued,
$\sigma(i^*)\ne 0$ for some $i^*\in\N$.
Suppose that
$$
q_{ij}(\phi)
=
\begin{cases}
1+\dfrac{c_i}{\|\phi\|+1}&\text{ if } j\in\{1, 2, i+1\}, j\ne i \text{ or } i=1, j=3.\\
-2-\dfrac{2c_i}{\|\phi\|+1}&\text{ if } i=j\in\{1,2\}\\
-3-\dfrac{3c_i}{\|\phi\|+1}&\text{ if } i=j>2\\
\end{cases}
$$
where $c_i$ are positive and bounded.
Thus, the limit at infinity of $Q(\phi)$ is
$$
\hat Q=\left(\begin{array}{cccccc}
-2&1&1&0&0&\cdots\\
1&-2&1&0&0&\cdots\\
1&1&-3&1&0&\cdots\\
1&1&0&-3&1&\cdots\\
\vdots&\vdots&\vdots&\vdots&\vdots&\ddots\\
\end{array}\right).
$$
Applying \cite[Proposition 3.3]{WA},
we can show that the Markov chain
$\hat\alpha(t)$ with generators $\hat Q$
is strongly ergodic.
Solving the system
$$\bnu \hat Q=0, \sum\bnu_i=1$$
we obtain that the invariant measure of $\hat\alpha(t)$
is
$\nu_1=\nu_2=\dfrac{1}{3}, \nu_k=\dfrac{2}{3^{k-1}}, k\geq 3.$
If
$\sum_i \theta(i)\nu(i)<0$
then the system is positive recurrent.
}\end{exm}

\begin{exm}{\rm
Suppose $Q(\phi)\equiv\hat Q$
is the generator of a strongly ergodic Markov chain
with invariant probability measure $\bnu=(\nu_1,\nu_2,\dots)$
where $\nu_i>0$ for any $i\in\N$.
Let $A(i)$ be a $2\times 1$
vector and $B(i)$ a $2\times 2$
matrix; both of them are bounded in $i\in\N$.
Consider the diffusion $X(t)=(X_1(t),X_2(t))$ in $\R^2$
given by
$$
\begin{aligned}
dX(t)=&\Big(B(\alpha(t))X(t)+ \dfrac{A(\alpha(t))}{1+|X(t)|}
\Big)dt\\
&+\dfrac{|X_1(t)|+|X_2(t)|}{2+|X_1(t)|+|X_2(t)|}\Big(c_{1,\alpha(t)}X_1(t)dW_1(t),c_{2,\alpha(t)}X_2(t)dW_2(t)\Big)^\top.	
\end{aligned}
$$
In particular,
suppose that
$A(i)\ne 0$, $c_{1,i}, c_{2,i}\in\R\setminus\{0\}$ are bounded,
Then
 \bea \ad a_1(i):=\Lambda_{B(i)}, \\ 
\ad a_2(i):=\Lambda\Big(a(i)\Big)=c_{1,i}^2\vee c_{2,i}^2, \\
\ad a_3(i):=\sum_{j=1}^2 \rho_{\sigma_j(i)}^2= c_{1,i}^2\wedge c^2_{2,i}.\eea

Note that the diffusion is degenerate at $0$,
but the drift at $0$ is non-zero.
Thus, the positive recurrence of the system
can still be obtained if
$$\sum_i \nu_i\Big(a_1(i)+\dfrac{a_2(i)}2-a_3(i)\Big)<0.$$

}\end{exm}

\begin{exm}\label{ex2} {\rm
This example is motivated by a controlled stochastic dynamic system that is linear in the continuous state variable $x$.
To be more specific, consider the following
scalar switching diffusion
\begin{equation}\label{e1-ex2}
\begin{aligned}
d X(t)=&\big[C(\alpha(t))+A(\alpha(t))X(t)+B(i)u(t)\big]dt
+\sigma(\alpha(t))X(t)dW_1(t)+dW_2(t),
\end{aligned}
\end{equation}
where $A(i), B(i), C(i),\sigma(i), c_i$ are bounded.
$W_1, W_2$ are two independent Brownian motions.
Let $Q(\phi)=\wdt Q(|\phi(-r)|)$ where
$$
\wdt Q(x)=\left(\begin{array}{ccccc}
\frac{-x}{c_1+x}&\frac{x}{c_1+x}&0&0&\cdots\\
\frac{x}{c_2+x}&\frac{-2x}{c_2+x}&\frac{x}{c_2+x}&0&\cdots\\
\frac{x}{c_3+x}&0&\frac{-2x}{c_3+x}&\frac{x}{c_3+x}&\cdots\\
\vdots&\vdots&\vdots&\vdots&\vdots\\
\end{array}\right).
$$
and $c_i, i=1,2,\dots,$ are bounded positive constants.
Thus, the limit at infinity of $Q(\phi)$ is
$$
\hat Q=\left(\begin{array}{ccccc}
-1&1&0&0&\cdots\\
1&-2&1&0&\cdots\\
1&0&-2&1&\cdots\\
\vdots&\vdots&\vdots&\vdots&\vdots\\
\end{array}\right).
$$
By \cite[Proposition 3.3]{WA},
it is easy to verify that the Markov chain
$\hat\alpha(t)$ with generators $\hat Q$
is strongly ergodic.
Solving the system
$$\bnu \hat Q=0, \  \sum_i\bnu_i=1$$
we obtain that the invariant measure of $\hat\alpha(t)$
is
$$\{\nu_i: i \in \N\}=\{2^{-i}: 1\le i< \infty\}.$$
Let $u(t)=-L(i)X(t)$
we have
$$dX(t)=\big[v_i+[A(i)-B(i)L(i)]X(t)dt+\sigma(i)X(t)dW_1(t)+dW_2(t).$$
Suppose that we can control only on mode $i=1$.
Thus, $L(i)=0$ if $i\geq 2$.
If $B(1)>0$,
choosing $L(i)$ sufficiently large that
$$A(1)-B(1)L(1)-\dfrac12\sigma(1)<-\sum_{i=2}^n 2^{i-1}\Big[A(i)-\dfrac12\sigma(i)\Big]$$
 then the controlled system is weakly stabilized.
}\end{exm}

\section{Final Remarks}\label{sec:rem}
This work developed more verifiable conditions on
 on recurrence and positive recurrence and related issues for switching diffusions with the switching taking values in a countably infinite set and depending on the history of the state process. It also developed feedback strategies for weakly stablization.
 For feedback controls of diffusions, we refer to the excellent work of Yong and Zhou \cite{YongZ}.
 
 For the systems under consideration, there 
 are numerous applications and potential applications. As one particular example, we mention the recent work \cite{HNY17}. 
 Extending the effort of treating stochastic population growth in spatially heterogeneous environments to include random environment
 modeled by stochastic switching processes is both theoretically interesting and practically useful. Moreover, effort may also be directed to treating regime-switching models in
  related work on other ecological and biological applications; see for example \cite{DDNY16,DY1a} among others.

\appendix
\section{Appendix}\label{sec:apd}
\subsection{Preliminary Results}
Recently, a functional It\^o formula was developed in \cite{Dup09}, which encouraged subsequent development (for example, \cite{CF}). We briefly recall the main idea in what follows.

Now we state the functional It\^o formula for our process (see \cite{CF} for more details).
Let $\DD$ be the space of cadlag functions $f:[-r,0]\mapsto\R^n$.
For $\phi\in\DD$, we define horizontal and vertical perturbations for $h\ge 0$ and $y\in \R^n$ as
$$
\phi_h(s)=
\begin{cases}\phi(s+h)\, \text{ if }\, s\in[-r,-h],\\
 \phi(0) \, \text{ if }\,s\in[-h,-0],
\end{cases}
$$
and
$$
\phi^y(s)=
\begin{cases}\phi(s)\, \text{ if }\, s\in[-r,0),\\
 \phi(0)+y,
\end{cases}
$$
respectively.
Let $V:\DD\times\N\mapsto\R$.
The horizontal derivative at $(\phi,i)$ and vertical partial derivative of $V$ are defined as
\begin{equation}\label{e.dt}
V_t(\phi,i)=\lim\limits_{h\to0} \dfrac{V(\phi_h,i)-V(\phi)}h
\end{equation}
and
\begin{equation}\label{e.dx}
\partial_i V(\phi,i)=\lim\limits_{h\to0} \dfrac{V(\phi^{he_i},i)-V(\phi)}h
\end{equation}
if these limits exist.

Consider functions of the form
 $$V(\phi,i)=f_1(\phi(0),i)+\int_{-r}^0g(t,i)f_2(\phi(t),i)dt,$$
where $f_2(\cdot,\cdot):\R^\nz\times\N\mapsto\R$ is a continuous function and $f_1(\cdot,\cdot):\R^\nz\times\N\mapsto\R$ is a function that is twice continuously differentiable in the first variable
and $g(\cdot,\cdot):\R_+\times\N\mapsto\R$ be a continuously differentiable function in the first variable.
Then at $(\phi,i)\in\C\times\N$ we have
$$V_t(\phi,i)=g(0,i)f_2(\phi(0),i)-g(-r,i)f_2(\phi(-r),i)-\int_{-r}^0 f_2(\phi(t),i)dg(t,i),$$
$$\partial_k V(\phi,i)=\dfrac{\partial f_1}{\partial x_k}(\phi(0),i),\qquad \partial_{kl} V(\phi,i)=\dfrac{\partial^2 f_1}{\partial x_k\partial x_l}(\phi(0),i).$$

Let $V(\cdot,\cdot)\in\BF$, we define the operator
\begin{equation}\label{e:LV}
\begin{aligned}
\LL V(\phi, i)
=&V_t(\phi,i)+
\sum_{k=1}^nb_k(\phi(0),i)V_k(\phi,i)
+\dfrac12
\sum_{k,l=1}^na_{kl}(\phi(0),i)V_{kl}(\phi,i)\\
&+\sum_{j=1,j\ne i}^\infty q_{ij}(\phi)\big[V(\phi,j)-V(\phi,i)\big],
\end{aligned}
\end{equation}
for any bounded stopping time $\tau_1\leq\tau_2$,
we have the functional It\^o formula:
\begin{equation}\label{f.Ito}
\E V(X_{\tau_2},\alpha(\tau_2))=\E V(X_{\tau_1},\alpha(\tau_1))+\E\int_{\tau_1}^{\tau_2}\LL V(X_s,\alpha(s))ds
\end{equation}
if the expectations involved exist. Equation
\eqref{f.Ito} is obtained by applying the functional It\^o formula for general semimartingales
given in \cite{CF1,CF}
specialized to our processes.

\subsection{Proofs of Lemmas}

\begin{proof}[Proof of Lemma {\rm\ref{lm3.1}}]
Since there is a $K>0$ such that
$$\|b(x,i)\|+\|\sigma(x,i)\|\leq K(1+|x|).$$
Using this,
we obtain the lemma from standard arguments (e.g., \cite{XM, ZY2}).
\end{proof}

\begin{proof}[Proof of Lemma {\rm\ref{lm3.2}}]
In view of \eqref{e1-thm3.1} and \eqref{e3-ex2},
there
are $H_4>0$ and $H_5>0$ such that
$$\LL_i \ln |x|\leq H_5, \text{ and } |\wdt\sigma(x,i)|\leq H_5 \text{ if } |x|\geq H_4.$$
Since $\LL_i\ln|x|$ and $\wdt\sigma(x,i)$ are bounded
it is easy to prove that, for any $T>0$, there exists a $K_4>0$ such that
$$
\begin{aligned}
\E_{\phi,i}\sup_{t\in[0,T]}\bigg|
\int_0^{t\wedge\tau_{H_4}}\Big[&\LL\ln|X(s)| ds+\wdt\sigma(X(s),\alpha(x))dW(s)ds\Big]\bigg|\leq K_4
\end{aligned}
$$
for any $(\phi,i)\in\C\times\N$.
Thus, for any $\eps>0$,
$$
\begin{aligned}
\PP_{\phi,i}\left\{\sup_{t\in[0,T]}\bigg|
\int_0^{t\wedge\tau_{H_4}}\Big[\LL\ln|X(s)| ds+\wdt\sigma(X(s),\alpha(x))dW(s)ds\Big]\bigg|> K_4\eps^{-1}\right\}<\eps.
\end{aligned}
$$
As a result,
\begin{equation}\label{e3-lm3.2}
\PP_{\phi,i}\left\{\inf_{t\in[0,T]}\ln X(t\wedge\tau_{H_4})>\ln |x|-K_4\eps^{-1}\right\}>1-\eps.
\end{equation}
For any $K_2>0$,
let $K_3$ so large that $\ln K_3>\ln (K_2\vee H_4)+K_6\eps^{-1}$.
We obtain from \eqref{e3-lm3.2}
that
$$
\PP_{\phi,i}\left\{\sup_{t\in[0,T]}\ln |X(t)|>\ln (K_2\vee H_4)\right\} >1-\eps
$$
if $|\phi(0)|>K_3$,
which completes the proof.
\end{proof}

\begin{proof}[Proof of Lemma {\rm\ref{lm3.3}}]
Let $\eta(k)=0$ for $k\leq k_0$.
By It\^o's formula, we have for $i>k_0$ that
$$
\begin{aligned}
\E_{\phi,i}\eta(\alpha(t\wedge\zeta))
=&\eta_i+\E_{\phi,i}\int_0^{t\wedge\zeta}\sum_{j\in\N}
q_{\alpha(s),j}\eta(j)ds\\
\leq& \eta_i+\E_{\phi,i}\int_0^{t\wedge\zeta}\sum_{j>k_0}q_{\alpha(s),j}\eta(j)\\
\leq& \eta_i-\E_{\phi,i}(t\wedge\zeta)
\end{aligned}
$$
thus
$\E_{\phi,i}(t\wedge\zeta)\leq\eta_i$
for any $t\geq0$.
Letting $t\to\infty$ we have
$\E_{\phi,i}\zeta\leq \eta_i\leq\sup_{i\in\N}\eta_i<\infty$.
My the assumptions \eqref{asp3.1} and \eqref{asp3.2},
we also have
$\sum_{j>k_0} \hat q_{ij}(\phi)\eta_j\leq -1\,\text{ for any }i>k_0.$
By \cite[Proposition 6.3.3]{WA},
$\alpha(t)$ is strongly ergodic.
\end{proof}

\begin{proof}[Proof of Lemma {\rm\ref{lm3.4}}]
By the basic coupling method (see e.g., \cite[p. 11]{Chen}),
we can consider the joint process $(X(t),\alpha(t),\hat\alpha(t))$
as a switching diffusion
where the diffusion $X(t)\in\R^\nz$ satisfies
satisfying
\begin{equation}\label{e1-a3}
dX(t)=b(X(t), \alpha(t))dt+\sigma(X(t),\alpha(t))dw(t)
\end{equation}
and the
switching part $(\alpha(t),\hat\alpha(t))\in\N\times\N$ has the generator $\wdt Q(X_t)$ which is defined by
$$\begin{aligned}
\wdt  Q(\phi)\wdt  f(k,l)=&\sum_{j,i\in\N}\wdt q_{(k,l)(j,i)}(\phi)\Big(\wdt  f(j,i)-\wdt  f(k,l)\Big)\\
=&\sum_{j\in\N}[q_{kj}(\phi)-\hat q_{lj}]^+(\wdt  f(j,l)-\wdt  f(k,l))\\
&+\sum_{j\in\N}[\hat q_{lj}-q_{kj}(\phi)]^+(\wdt  f(k,j)-\wdt  f(k,l))\\
&+\sum_{j\in\N}[q_{kj}(\phi)\wedge \hat q_{lj}(0)](\wdt  f(j,j)-\wdt  f(k,l)).
\end{aligned}
$$
Let $\vartheta=\inf\{t\geq0: \alpha(t)\ne\hat\alpha(t)\}$.
Define $\tilde g:\Z\times\Z\mapsto\R$ by $\wdt g(k,l)=\1_{\{k=l\}}$.
By the definition of the function $\wdt g$,  we have
\begin{equation}\label{e3-a3}
\begin{aligned}
\wdt  Q(\phi)\wdt  g(k,k)=&\sum_{j\in\N, j\ne k}[q_{kj}(\phi)-\hat q_{kj}]^++\sum_{j\in\N, j\ne k}[\hat q_{kj}-q_{kj}(\phi)]^+\\
=&\sum_{j\in\N, j\ne k}|q_{kj}(\phi)-\hat q_{kj}|=:\chi(\phi, k).
\end{aligned}
\end{equation}
For any $\eps>0$,
let $H>0$ such that
$\sup_{(\phi,k)\in B_H\times\N} \chi(x,k)<\frac{\eps}{2T}$.
where $B_H:=\{\phi\in\C:|\phi(t)|\geq H, t\in[-r,0]\}.$
Applying  It\^o's formula and noting that $\alpha(t)=\hat \alpha(t), t\leq\vartheta$,  we obtain that
\begin{equation}\label{e4-a3}
\begin{split}
\PP_{\phi,i,i}\{\vartheta\leq T\wedge\tau_H\}=& \E_{\phi,i,i} \wdt g\Big(\alpha(\vartheta\wedge T\wedge\tau_H),\hat \alpha(\vartheta\wedge T\wedge\tau_H)\Big)\\
=&\E_{\phi,i,i}\int_0^{\vartheta\wedge T\wedge\tau_H}\wdt Q(X_t) \wdt  g(\alpha(t),\hat \alpha(t))dt\\
=&\E_{\phi,i,i} \int_0^{\vartheta\wedge T\wedge\tau_H}\chi(\phi, \alpha(t))dt\\
\leq& T\sup_{(\phi,i)\in B_H\times\N}\chi(\phi,k)\leq \dfrac\eps2.
\end{split}
\end{equation}
Thus,
in view of Lemma \ref{lm3.2},
there is $\delta>0$ such that
$\PP_{\phi,i,i}\{\tau_H\leq T\}\leq \dfrac\eps2$.
This together with \eqref{e4-a3} leads to
$$
\PP_{\phi,i,i}\{\vartheta\wedge\tau_H\leq T\}\leq \PP_{\phi,i,i}\{\vartheta\leq T\wedge\tau_H\}+\PP_{\phi,i,i}\{\tau_H\leq T\}\leq \eps.$$

We have that
$$
\begin{aligned}
\Big|\E_{\phi,i}f(\alpha(t))-\E_{i}f(\hat\alpha(t))\Big|
=&\Big|\E_{\phi,i,i}\Big[f(\alpha(t))-f(\hat\alpha(t))\Big]\Big|\\
=&\Big|\E_{\phi,i,i}\1_{\{\vartheta\wedge\tau_H\leq t\}}\Big[f(\alpha(t))-f(\hat\alpha(t))\Big]\Big|\\
\leq &2M_f\PP_{\phi,i,i}\{\vartheta\wedge\tau_H\leq t\}\leq 2M_f\eps,\text{ for } t\in[0,T].
\end{aligned}
$$
where $$M_f=\sup_{i\in\N} |f(i)|.$$
The lemma is therefore proved.
\end{proof}

\begin{proof}[Proof of Lemma {\rm\ref{lm3.4}}]
This lemma concerns basic properties of the Laplace transform.
It can be proved by straightforward arguments in calculus. The proof is therefore omitted.
\end{proof}

\end{document}